\documentclass[12pt, reqno]{amsart}

\usepackage{enumerate,amssymb}

\newtheorem{theorem}{Theorem}[section]

\newtheorem{thm}{Theorem}

\theoremstyle{remark}
\newtheorem{remark}[theorem]{Remark}
\theoremstyle{definition}

\theoremstyle{plain}

\newcommand{\muT}[1]{\liminf_{T\to\infty}\frac{1}{T}\mu\left\{\tau\in[0,T]: #1\right\}>0}
\newcommand{\norm}[2]{\left\Vert#1\right\Vert_{#2}}
\renewcommand{\Re}{\operatorname{Re}}

\begin{document}

\title[Corrigendum for ``The strong recurrence for non-zero rational parameters'']{Corrigendum for ``The generalized strong recurrence for non-zero rational parameters'' Archiv der Mathematik 95 (2010), 549--555}

\author{Takashi Nakamura}
\address{Department of Mathematics Faculty of Science and Technology \\Tokyo University of Science Noda, CHIBA 278-8510 JAPAN}
\email{nakamura\_takashi@ma.noda.tus.ac.jp}

\author{{\L}ukasz Pa\'nkowski}
\address{Faculty of Mathematics and Computer Science, Adam Mickiewicz University, Umultowska 87, 61-614 Pozna\'{n}, POLAND}
\email{lpan@amu.edu.pl}

\thanks{The first author was partially supported by JSPS Grants 21740024. The second author was partially supported by the grant no. N N201 6059 40 from National Science Centre.}



\begin{abstract}
In the present paper, we prove that self-approximation of $\log \zeta (s)$ with $d=0$ is equivalent to the Riemann Hypothesis. Next, we show self-approximation of $\log \zeta (s)$ with respect to all nonzero real numbers $d$. Moreover, we partially filled a gap existing in `The strong recurrence for non-zero rational parameters'' and prove self-approximation of $\zeta(s)$ for $0 \ne d=a/b$ with $|a-b|\ne 1$ and $\gcd(a,b)=1$.
\end{abstract}

\maketitle

\section{Introduction}
 
In 1981 Bagchi \cite{BagchiRecurrence} discovered that the following almost periodicity holds in the critical strip if and only if the Riemann Hypothesis is true. To state it, let $\mu \{A\}$ stand for the Lebesgue measure of a measurable set $A$, $D := \{ s \in {\mathbb{C}} : 1/2 < \Re (s) <1\}$ and $H(K)$ denote the space of non-vanishing continuous functions on a compact set $K$, which are analytic in the interior, equipped with the supremum norm $\norm{\cdot}{K}$. Then Bagchi's result can be formulated as follows (see also \cite[Section 8]{Steuding1}).
\begin{thm}\label{th:bagchiA}
The Riemann Hypothesis holds if and only if, for every compact set $K\subset D$ with connected complement and for every $\varepsilon > 0$, we have
\begin{equation}\label{eq:SelfSimZeta}
\muT{\norm{\zeta(s+i\tau) - \zeta(s)}{K} < \varepsilon}.
\end{equation}
\end{thm}

In 2010 Nakamura \cite{Nakamura1} showed the following property which might be called \textit{self-approximation} of the Riemann zeta function. 
\begin{thm}
For every algebraic irrational number $d\in\mathbb{R}$, every compact set $K\subset D$ with connected complement and every $\varepsilon>0$, we have
\begin{equation*}
\muT{\norm{\zeta(s+id\tau) - \zeta(s+i\tau)}{K} < \varepsilon}.
\end{equation*}
\end{thm}
Note that the self-approximation with respect to almost all real numbers $d$ was also verified in \cite{Nakamura1}. Afterwards, Pa\'nkowski \cite{PankowskiRec} showed the above result for any irrational number $d$ whereas Garunk\v{s}tis \cite{Garu} and Nakamura \cite{Nakamura3} investigated the self-approximation for non-zero rational numbers, independently. Unfortunately, the papers \cite{Garu} and \cite{Nakamura3} contain a gap in the proof of the main theorem, so actually their methods work only for the logarithm of the Riemann zeta function (see Remark \ref{rem:zetalogzeta}). 

In this paper, we prove that self-approximation of $\log \zeta (s)$ with $d=0$ is equivalent to the Riemann Hypothesis in Theorem \ref{th:salogz1}. Next, we show self-approximation of $\log \zeta (s)$ for all nonzero real numbers $d$ in Theorem \ref{th:salogz2}. Moreover, we partially filled the gap mentioned above and prove self-approximation of $\zeta(s)$ for $0 \ne d=a/b$ with $|a-b|\ne 1$ and $\gcd(a,b)=1$ in Theorem \ref{th:sazabn1}.

\section{Self-approximation of $\log \zeta (s)$}

Firstly, we show the following theorem which is an analogue of Theorem \ref{th:bagchiA}. 
\begin{theorem}\label{th:salogz1}
The Riemann Hypothesis holds if and only if, for every compact set $K\subset D$ with connected complement and for every $\varepsilon > 0$, we have
\begin{equation}\label{eq:SelfSimLogZeta1}
\muT{\norm{\log \zeta(s+i\tau) - \log \zeta(s)}{K} < \varepsilon}.
\end{equation}
\end{theorem}
\begin{proof}
If the Riemann hypothesis is true we can apply Voronin's universality theorem. 

Suppose that there exists a zero $\xi \in D$ of $\zeta (s)$. Put $K_{\varepsilon} := \{ s \in {\mathbb{C}} : |s-\xi| \le \varepsilon \} \subset D$. Now assume that for a neighborhood $K_{\varepsilon}$ of $\xi$ the following relation holds:
\begin{equation}
\| \log \zeta (s+i\tau) - \log \zeta (s) \|_{K_{\varepsilon}} < \varepsilon .
\label{eq:lepf1}
\end{equation}
If a zero $\rho$ of $\zeta (s)$, where $\rho \in K_{\varepsilon} (\tau) := \{ s \in {\mathbb{C}} : |s-\xi-i\tau| \le \varepsilon \}$ does not exist, then the function $\log \zeta (s+i\tau)$ is analytic in the interior of $K_{\varepsilon}$ and bounded on $K_{\varepsilon}$. This contradicts to the above inequality. Hence a zero $\rho$ of $\zeta (s)$ exists in $K_{\varepsilon} (\tau)$. With regard to (\ref{eq:lepf1}) and the definition of $K_{\varepsilon} (\tau)$ the zeros $\xi$ and $\rho$ are intimately related; more precisely, $|\rho- \xi - i\tau| < \varepsilon$. Thus two different shifts $\tau_1$ and $\tau_2$ can lead to the same zero $\rho$, but their distance is bounded by $|\tau_1-\tau_2| < 2\varepsilon$. Therefore we obtain this lemma by modifying the proof of \cite[Theorem 8.3]{Steuding1} and using the classical Rouch\'e theorem.
\end{proof}

The reasoning of \cite[Theorem 1.1]{PankowskiRec} and \cite[Theorem 1]{Garu} can be easily applied to prove self-approximation of $\log \zeta$. 
\begin{theorem}\label{th:salogz2}
For every real number $d \ne 0$, every compact set $K\subset D$ with connected complement and for every $\varepsilon > 0$, it holds that
\begin{equation}\label{eq:SelfSimLogZeta2}
\muT{\norm{\log \zeta(s+i\tau) - \log \zeta(s+id\tau)}{K} < \varepsilon}.
\end{equation}
\end{theorem}

\begin{remark}\label{rem:zetalogzeta}
In fact, in \cite{Garu} and \cite{Nakamura3} it was claimed that the above theorem holds even for the Riemann zeta function instead of the logarithm of $\zeta(s)$. Unfortunately, the proofs of main results in \cite[Theorem 1]{Garu} and \cite[Corollary 1.2]{Nakamura3} are not sufficient, since it was only shown that $|\zeta(s+di\tau)/\zeta(s+i\tau)-1|<\varepsilon$. Obviously we have
$$
|\zeta (s+i\tau) - \zeta (s+id\tau)| = |\zeta (s+i\tau)| |\zeta (s+id\tau) / \zeta (s+i\tau) -1|.
$$
So in order to prove self-approximation of $\zeta(s)$ it should have been proved that $\zeta (s+i\tau)$ is not too large, namely we need the inequality $|\zeta (s+i\tau)| < |\zeta(s+id\tau)/\zeta(s+i\tau)-1|^{-1}$ which is not necessarily satisfied. 
\end{remark}

\section{Self-approximation of $\zeta (s)$}

In the following theorem we partially fix the gap existing in \cite[Theorem 1]{Garu} and \cite[Corollary 1.2]{Nakamura3} and prove self-approximation of $\zeta(s)$ in the case $0\ne d=a/b\in\mathbb{Q}$ with $|a-b|\ne 1$ and $\gcd(a,b)=1$.
\begin{theorem}\label{th:sazabn1}
For every $0 \ne d=a/b$ with $|a-b| \ne 1$ and $\gcd(a,b)=1$, every compact set $K\subset D$ and for every $\varepsilon > 0$, we have
\begin{equation}\label{eq:SelfSimZetaabn1}
\muT{\norm{\zeta(s+i\tau) - \zeta(s+id\tau)}{K} < \varepsilon}.
\end{equation}
\end{theorem}
\begin{proof}
First of all, note that it suffices to consider the case $a\ne b$, or equivalently $d\ne 1$.

Let us take $\omega(p_m) := \exp(2\pi i m/(a-b))$, where $p_m$ denotes the $m$-th prime number. Then we have $\omega^a(p)= \omega^b(p)$ since $\omega^{a-b} (p)=1$. Firstly, we show that 
\begin{equation}
|\zeta_z (s,\omega^c)| \leq \exp \left( \frac{7(1-z^{1-2\sigma})}{2\sigma-1} + |a-b| \bigl( z^{-\sigma} +
2|s| (1-z^{-\sigma}) \bigr) \right),
\label{eq:ztomesti}
\end{equation}
where $c=a,b$ and $\zeta_z (s,\omega^c) := \prod_{p \leq z} ( 1-\omega(p)^cp^{-s})^{-1}$. In order to prove it let us consider the function $-\sum_{p\leq z}\log(1-\omega(p)^cp^{-s})$. Then
\begin{align*}
-\sum_{p\leq z}\log\left(1-\frac{\omega(p)^c}{p^s}\right) &= 
\sum_{p\leq z}\sum_{k=1}^\infty \frac{\omega(p)^{ck}}{kp^{ks}} = 
\sum_{p\leq z}\frac{\omega(p)^c}{p^s} + \sum_{p\leq z}\sum_{2\leq k} \frac{\omega(p)^{ck}}{kp^{ks}}.
\end{align*}
Let us estimate the latter sum on the right hand side. For $\sigma>1/2$ one has
\begin{align*}
\left|\sum_{p\leq z}\sum_{2\leq k} \frac{\omega(p)^{ck}}{kp^{ks}} \right|&\leq \sum_{p\leq z}\sum_{2\leq k} \frac{1}{p^{k\sigma}} = \sum_{p\leq z}\frac{1}{p^{2\sigma}-p^\sigma} \leq 7\sum_{p\leq z}\frac{1}{p^{2\sigma}} \\
&\leq 7 \sum_{2 \le n \leq z}\frac{1}{n^{2\sigma}} \leq 7\int_1^z t^{-2\sigma}dt = 
\frac{7(1-z^{1-2\sigma})}{2\sigma-1} .
\end{align*}
To consider the former sum, we put
$$\omega (n) := \begin{cases}\omega (p) &\text{if \ }n=p\leq z,\\0 &\text{otherwise,}\end{cases} \qquad
\Omega_z := \sum_{n=1}^z \omega(n)^c.
$$
Then, by partial summation, we obtain 
$$
\sum_{p\leq z}\frac{\omega(p)^c}{p^s} = \sum_{n=1}^z \frac{\omega(n)^c}{n^s} = 
\frac{\Omega_z}{z^s} - \sum_{n=1}^{z-1}\Omega_n\left(\frac{1}{(n+1)^s}-\frac{1}{n^s}\right) 
= \frac{\Omega_z}{z^s} + s\sum_{n=1}^{z-1}\Omega_n\int_n^{n+1}t^{-s-1}dt .
$$
Let us notice that $\omega(p)^c$ is a nontrivial root of unity, since $\gcd(a,b)=1$ implies $a-b\nmid a,b$. Hence we have 
\begin{align*}
\left| \sum_{p\leq z}\frac{\omega(p)^c}{p^s} \right| \ll 
z^{-\sigma}+|s|\int_1^z t^{-\sigma-1}dt = z^{-\sigma} +\frac{|s|}{\sigma} \left(1-z^{-\sigma}\right) \leq 
\frac{1}{z^\sigma}+2|s|\left(1-\frac{1}{z^\sigma}\right),
\end{align*}
where the constant in symbol $\ll$ is equal to $|a-b|$. Therefore we obtain (\ref{eq:ztomesti}).

Now it suffices to follow the steps of the proof of \cite[Theorem 1]{Garu} and the following fact. Let $\| x \|$ give the the distance from a real number $x$ to the the nearest integer. Then the set of positive real numbers $\tau$ satisfying
\[
\max_{p_m \leq z} \left\Vert\tau\frac{\log p_m}{2\pi}-\frac{m}{a-b}\right\Vert<\delta
\]
has a positive density for every positive $\delta$ by the Kronecker approximation theorem. Thus it holds for sufficiently large $z$
\begin{equation*}
\| \log \zeta (s+ic\tau) - \log \zeta_z (s, \omega^c) \|_K < \varepsilon , \quad c=a,b .
\end{equation*}
This completes the proof.
\end{proof}



\begin{thebibliography}{10}

\bibitem {BagchiRecurrence} B. Bagchi, \textit{Recurrence in topological dynamics and the Riemann hypothesis}, Acta Math. Hung. \textbf{50} (1987), 227-240.

\bibitem{Garu} R.~Garunk\v{s}tis, {\textit{Self-approximation of Dirichlet $L$-functions,}} J. Number Theory \textbf{131} (2011), no. 7, 1286--1295.

\bibitem{Nakamura1} T.~Nakamura, {\textit{The joint universality and the generalized strong recurrence for Dirichlet {\it{L}}-functions,}} Acta Arith. {\bf{138}} (2009), no.~4, 357--362. 

\bibitem{Nakamura3} T.~Nakamura, {\textit{The generalized strong recurrence for non-zero rational parameters,}} Archiv der Mathematik {\bf{95}} (2010), 549--555.

\bibitem{PankowskiRec} \L.~Pa\'{n}kowski, \textit{Some remarks on the generalized strong recurrence for L-functions}, New directions in value-distribution theory of zeta and $L$-functions, 305--315, Ber.~Math., Shaker Verlag, Aachen, 2009.

\bibitem{Steuding1} J.~Steuding, \textit{Value-Distribution of L-functions}, Lecture Notes in Mathematics, 1877, Springer, Berlin (2007).

\end{thebibliography}
\end{document}